\documentclass[a4paper, 10pt, twoside]{article}

\usepackage{amsmath, amscd, amsfonts, amssymb, amsthm, latexsym, url, color, todonotes, bm, framed, rotating, enumerate} 
\usepackage{graphicx}
\usepackage[left=1in, right=1in, top=1.2in, bottom=1in, includefoot, headheight=13.6pt]{geometry}
\usepackage{mathtools}
\input{xy}
\xyoption{all}

\usepackage[colorlinks,breaklinks=true]{hyperref}
\usepackage[figure,table]{hypcap}
\hypersetup{
	bookmarksnumbered,
	pdfstartview={FitH},
	citecolor={black},
	linkcolor={black},
	urlcolor={black},
	pdfpagemode={UseOutlines}
}
\makeatletter
\newcommand\org@hypertarget{}
\let\org@hypertarget\hypertarget
\renewcommand\hypertarget[2]{%
  \Hy@raisedlink{\org@hypertarget{#1}{}}#2%
} 
\makeatother 

\newtheorem{theorem}{Theorem}[section]
\newtheorem{lemma}[theorem]{Lemma}
\newtheorem{corollary}[theorem]{Corollary}
\newtheorem{proposition}[theorem]{Proposition}

\theoremstyle{definition}

\newtheorem{remark}[theorem]{Remark}

\newcommand{\xysquare}[8]{
\[\xymatrix{
#1 \ar@{#5}[r] \ar@{#6}[d] & #2 \ar@{#7}[d]\\
#3 \ar@{#8}[r] & #4
}\]
}

\newcommand{\al}{\alpha}
\newcommand{\bb}{\mathbb}

\newcommand{\blob}{\bullet}

\newcommand{\comment}[1]{}

\newcommand{\into}{\hookrightarrow}

\newcommand{\isoto}{\stackrel{\simeq}{\to}}
\newcommand{\Isoto}{\stackrel{\simeq}{\longrightarrow}}

\newcommand{\onto}{\twoheadrightarrow}
\newcommand{\op}{\operatorname}

\renewcommand{\phi}{\varphi}
\newcommand{\quis}{\stackrel{\sim}{\to}}
\newcommand{\res}{\overline}
\newcommand{\roi}{\mathcal{O}}

\newcommand{\sub}[1]{{\mbox{\rm \scriptsize #1}}}

\newcommand{\To}{\longrightarrow}
\newcommand{\ul}[1]{\underline{#1}}

\newcommand{\xto}{\xrightarrow}

\newcommand{\THH}{T\!H\!H}

\newcommand{\TR}{T\!R}
\newcommand{\TC}{T\!C}

\renewcommand{\cal}{\mathcal}
\renewcommand{\hat}{\widehat}
\renewcommand{\frak}{\mathfrak}
\newcommand{\indlim}{\varinjlim}
\renewcommand{\tilde}{\widetilde}

\renewcommand{\ker}{\operatorname{Ker}}
\renewcommand{\projlim}{\varprojlim}

\DeclareMathOperator{\Char}{char}

\DeclareMathOperator{\dlog}{dlog}

\DeclareMathOperator{\Fil}{Fil}
\DeclareMathOperator{\Frac}{Frac}

\DeclareMathOperator{\Spec}{Spec}

\newcommand{\KH}{K\!H}
\newcommand{\OO}{\mathcal O}

\newcommand{\p}{\mathfrak p}

\newcommand{\dotimes}{\otimes^{\bb L}}
\newcommand{\xTo}[1]{\stackrel{#1}{\To}}

\usepackage{fancyhdr}

\usepackage{sectsty}
\sectionfont{\Large\sc\centering}
\chapterfont{\large\sc\centering}
\chaptertitlefont{\LARGE\centering}
\partfont{\centering}

\begin{document}
\itemsep0pt

\title{$K$-theory of valuation rings}

\author{Shane Kelly and Matthew Morrow}

\date{}

\maketitle

\begin{abstract}
We prove several results showing that the algebraic $K$-theory of valuation rings behave as though such rings were regular Noetherian, in particular an analogue of the Geisser--Levine theorem. We also give some new proofs of known results concerning cdh descent of algebraic $K$-theory.
\end{abstract}

\tableofcontents

\section{Introduction and statements of main results}
Recall that a valuation ring is an integral domain $\roi$ with the property that, given any elements $f,g\in\roi$, either $f\in g\roi$ or $g\in f\roi$. Despite rarely being Noetherian, valuation rings surprisingly often behave like regular Noetherian rings. The theme of this article is to explore the extent to which this is reflected in their algebraic $K$-theory. The main results are largely independent of one another.

The primary new result is the following calculation of the $p$-adic $K$-theory of valuation rings of characteristic $p$, analogous to Geisser--Levine's theorem for regular Noetherian local rings \cite{GeisserLevine2000}:

\begin{theorem}\label{theorem_GL_intro}
Let $\roi$ be a valuation ring containing $\bb F_p$, and $n\ge0$. Then $K_n(\roi)$ is $p$-torsion-free and there is a natural isomorphism $K_n(\roi)/p^r\isoto W_r\Omega^n_{\roi,\sub{log}}$ given by $\dlog[\cdot]$ on symbols.
\end{theorem}

Here $W_r\Omega^n_{\roi,\sub{log}}$, often denoted by $\nu_r^n(\roi)$, is the subgroup of the de Rham--Witt group $W_r\Omega^n_\roi$ generated by dlog forms; see Remark \ref{remark_dlog} for further information.

Let us sketch the ideas of the proof of Theorem \ref{theorem_GL_intro}. By a recent result of the second author joint with Clausen--Mathew \cite{ClausenMathewMorrow}, the trace map $\op{tr}:K_n(-;\bb Z/p^r\bb Z)\to\TC_n(-;\bb Z/p^r\bb Z)$ is naturally split injective on the category of local $\bb F_p$-algebras. The proof of this injectivity reduces Theorem \ref{theorem_GL_intro} to showing that the topological Hochschild and cyclic homologies of $\roi$ behave as though it were a regular Noetherian local $\bb F_p$-algebra, which in turn comes down to controlling the derived de Rham and derived de Rham--Witt cohomologies of $\roi$. This may be done by means of deep results of Gabber--Ramero and Gabber on the cotangent complex and de Rham cohomology of valuation rings.

The second main result of the article is that the injectivity part of Gersten's conjecture remains true for valuation rings over a field; away from the characteristic this follows from the existence of alterations, while the torsion-freeness in Theorem \ref{theorem_GL_intro} provides the missing ingredient at the characteristic:

\begin{theorem}\label{theorem_gersten_intro}
Let $\roi$ be a valuation ring containing a field. Then $K_n(\roi)\to K_n(\Frac\roi)$ is injective for all $n\ge0$.
\end{theorem}

Presheaves which satisfy this injectivity property are studied in \cite{HuberKelly2018} where they are called ``torsion free''.

Another important feature of the $K$-theory of regular Noetherian rings is homotopy invariance. We observe that this extends to valuation rings:

\begin{theorem}\label{theorem_K_of_val_intro}
Let $\roi$ be a valuation ring. Then
\begin{enumerate}
\item $K_n(\roi)\to K_n(\roi[T_1,\dots,T_d])$ is an isomorphism for all $d\ge1$ and $n\in\bb Z$;
\item $K_n(\roi)\to \KH_n(\roi)$ is an isomorphism for all $n\in\bb Z$;
\item $K_n(\roi)=0$ for all $n<0$.
\end{enumerate}
\end{theorem}

In fact, the ``proof'' of Theorem \ref{theorem_K_of_val_intro} is simply the fact, observed by Gersten and Weibel in the 1980s, that stable coherence and finite global dimension are enough for the usual proof (from the regular Noetherian case) of the assertions to go through. Since these homological properties of valuation rings have been known since the 1970s, Theorem \ref{theorem_K_of_val_intro} has in principle been available for many years. But we are not aware of it having been previously noticed. We refer the interested reader to Remark \ref{remark_KH} for historical comments and relations to properties of algebraic $K$-theory in the cdh topology.

Remaining on the theme of cdh topology, we also give a new proof that $K[\tfrac1p]$ satisfies cdh descent in characteristic $p$ in the following sense:

\begin{theorem}\label{theorem_cdhp_intro}
Let
\[\xymatrix{
Y'\ar[d]\ar[r] & X'\ar[d] \\
Y\ar[r] & X
}\]
be an abstract blow-up square of schemes in which $p$ is nilpotent; assume that $X$ is quasi-compact quasi-separated and that the morphisms $X'\to X$ and $Y\into X$ are of finite presentation. Then the resulting square 
\[\xymatrix{
K(X)[\tfrac1p]\ar[d]\ar[r] & K(X')[\tfrac1p]\ar[d] \\
K(Y)[\tfrac1p]\ar[r] & K(Y')[\tfrac1p]
}\]
is homotopy cartesian.
\end{theorem}

Theorem \ref{theorem_cdhp_intro} is only conceivably new in the case that $X$ is not Noetherian (and one can probably even reduce to that case via Noetherian approximation), but the method of proof is completely different to what has appeared previously: we replace all the schemes by their perfections (having reduced to the case of $\bb F_p$-schemes) and then use results of Bhatt--Scholze \cite{BhattScholze2017} on the homological properties of perfect schemes to check that Tamme's recent excision condition \cite{Tamme2017} is satisfied. See Remark \ref{remark_pro_cdh} for some comparisons to related results.

\subsubsection*{Acknowledgements}
We thank Georg Tamme for correspondence about Theorem \ref{theorem_cdhp_intro} (in particular, he was already aware that the results of Bhatt--Scholze could be used to apply \cite{Tamme2017} to blow-ups of perfect schemes) and Ben Antieau for his comments.

\section{$p$-adic $K$-theory of valuation rings in characteristic $p$}
The goal of this section is to prove Theorem \ref{theorem_GL_intro} from the introduction. More generally, we work with $\bb F_p$-algebras $A$ satisfying the following smoothness criteria (Sm1--3), which we paraphrase by saying that $A$ is {\em Cartier smooth}:
\begin{enumerate}[(Sm1)]
\item $\Omega^1_{A}$ is a flat $A$-module;
\item $H_n(\bb L_{A/\bb F_p})=0$ for all $n>0$;
\item  the inverse Cartier map\footnote{For any $\bb F_p$-algebra $A$, recall that the inverse Cartier map $C^{-1}:\Omega^n_A\to H^n(\Omega^\blob_A)$ is the linear map satisfying $C^{-1}(fdg_1\wedge\cdots\wedge dg_n)=f^pg_1^{p-1}\cdots g_n^{p-1}dg_1\wedge\cdots\wedge dg_n$.} $C^{-1}:\Omega_{A}^n\to H^n(\Omega^\blob_{A})$ is an isomorphism for each $n\ge0$.
\end{enumerate}

The above criteria are of course satisfied if $A$ is a smooth $\bb F_p$-algebra, or more generally if $A$ is a regular Noetherian $\bb F_p$-algebra by N\'eron--Popescu desingularisation (an alternative proof avoiding N\'eron--Popescu may be found in the recent work of Bhatt--Lurie--Mathew \cite[Thm.~9.5.1]{BhattLurieMathew2018}), or even an ind-smooth $\bb F_p$-algebra. However, in the case in which $A$ is ind-smooth, the forthcoming results trivially reduce to the smooth case, where they are known.

More interestingly, any valuation ring $\roi$ of characteristic $p$ is Cartier smooth: the first two criteria are due to Gabber--Ramero \cite[Thm.~6.5.8(ii) \& Corol.~6.5.21]{GabberRamero2003}, while the Cartier isomorphism (Sm3) is a recent unpublished result of Gabber obtained by refining his earlier work with Ramero. (Of course, resolution of singularities in characteristic $p$ would imply that $\roi$ is ind-smooth, but we wish to avoid assuming resolution of singularities.) Therefore Theorem \ref{theorem_GL_intro} is a special case of the following more general result:

\begin{theorem}\label{theorem_GL_val_general}
Let $A$ be any Cartier smooth, local $\bb F_p$-algebra, and $n\ge0$. Then $K_n(A)$ is $p$-torsion-free and there is a natural isomorphism $K_n(A)/p^r\isoto W_r\Omega^n_{A,\sub{log}}$ given by $\dlog[\cdot]$ on symbols.
\end{theorem}

\begin{remark}
The extra degree of generality afforded by the axiomatic set-up of Theorem \ref{theorem_GL_val_general} is not without interest. It also applies, for example, to essentially smooth, local $S$-algebras, where $S$ is any perfect ring of characteristic $p$.
\end{remark}

\begin{remark}\label{remark_dlog}
We recall the definition of the group of dlog forms. Given a local $\bb F_p$-algebra $A$, we define $W_r\Omega^n_{A,\sub{log}}\subseteq W_r\Omega^n_A$ to be the subgroup generated by dlog forms $\tfrac{d[\al_1]}{[\al_1]}\wedge\cdots \wedge\tfrac{d[\al_n]}{[\al_n]}$, where $\al_1,\dots,\al_n\in A^\times$; i.e., the image of $\dlog[\cdot]:K_n^M(A)\to W_r\Omega^n_A$. If $A$ is not necessarily local, then $W_r\Omega^n_{A,\sub{log}}\subseteq W_r\Omega^n_A$ is defined to be the subgroup generated Zariski locally by such forms; the fact that this coincides with the more classical definition as the subgroup generated \'etale locally by dlog forms was proved in \cite[Corol.~4.2(i)]{Morrow_pro_GL}.

In particular, under the hypotheses of Theorem \ref{theorem_GL_val_general}, the canonical map $K_n^M(A)/p^r\to K_n(A)/p^r$ is surjective. We do not know whether it is an isomorphism (even under the usual assumption that $A$ has large enough residue field), as we do not know of any way to control Milnor $K$-theory using derived de Rham cohomology.
\end{remark}

\subsection{The de Rham--Witt complex of a Cartier smooth algebra}
We study in this subsection the behaviour of certain functors, such as de Rham--Witt groups, on Cartier smooth $\bb F_p$-algebras. In essence, we must show that all familiar results about the de Rham--Witt complex extend from smooth algebras to Cartier smooth algebras, thereby justifying the well-known maxim that it is the Cartier isomorphism which controls the structure of the de Rham--Witt complex.

Given a functor $\cal F:\bb F_p\op{-algs}\to\op{Ab}$ commuting with filtered colimits, let $\bb L\cal F:\bb F_p\op{-algs}\to D_{\ge0}(\op{Ab})$ be the left Kan extension of its restriction to finitely generated free $\bb F_p$-algebras. More classically, following Quillen's theory of non-abelian derived functors, $\bb L\cal F(A)$ is given by the simplicial abelian group $q\mapsto F(P_q)$ (or more precisely the associated chain complex via Dold--Kan), where $P_\blob\to A$ is any simplicial resolution of $A$ by free $\bb F_p$-algebras. 

Let $A\in \bb F_p\op{-algs}$ and $\cal F$ a functor as above. In general, the resulting canonial map $H_0(\bb L\cal F(A))\to\cal F(A)$ need not be an isomorphism; when this is true we will say that $\cal F$ is {\em right exact at $A$}. When $H_i(\cal F(A))=0$ for all $i>0$, we will say that $\cal F$ is {\em supported in degree $0$ at $A$}. Thus $\cal F$ is both right exact and supported in degree $0$ at a given $A$ if and only if the canonical map $\bb L\cal F(A)\to \cal F(A)[0]$ is an equivalence, in which case we say that $\cal F$ is {\em already derived at $A$}.

The following lemma is straightforward but will be used repeatedly:

\begin{lemma}\label{lemma}
Let $\cal F\to\cal G\to\cal H$ be a maps of functors $\bb F_p\op{-algs}\to\op{Ab}$ which commute with filtered colimits, such that $0\to \cal F(R)\to\cal G(R)\to\cal H(R)\to 0$ is exact for any finitely generated free $\bb F_p$-algebra $R$. Fix $A\in\bb F_p\op{-algs}$.

\begin{enumerate}
\item Suppose that two of $\cal F,\cal G,\cal H$ are already derived at $A$. Then so is the third if and only if the sequence $0\to \cal F(A)\to\cal G(A)\to\cal H(A)\to 0$ is exact.
\item Assume that $\cal F$ and $\cal G$ are already derived at $A$, and that $\cal F(A)\to\cal G(A)$ is injective. Then $\cal H$ is supported in degree $0$ at $A$.
\item Assume that $\cal F$ and $\cal H$ are supported in degree $0$ at $A$. Then $\cal G$ is supported in degree $0$ at $\cal G$.
\end{enumerate}
\end{lemma}
\begin{proof}
The hypotheses provide us with a fibre sequence $\bb L\cal F(A)\to\bb L\cal G(A)\to\bb L\cal H(A)$  mapping to $\cal F(A)\to\cal G(A)\to\cal H(A)$. The assertions are then all easy to check.
\end{proof}

\begin{proposition}\label{proposition_WOmega_already_derived}
The following functors are already derived at any Cartier smooth $\bb F_p$-algebra:
\[\Omega^n_{-}\quad Z\Omega^n_{-}\quad B\Omega^n_{-}\quad W_r\Omega^n_{-}.\] 
\end{proposition}
\begin{proof}
The case of $\Omega^n_-$ is classical: since the derived exterior power of an $A$-module may be computed in terms of any flat resolution, the functor $\Omega^n_-$ is already derived at any $\bb F_p$-algebra satsifying (Sm1--2).

We now consider the complexes of functors 
\begin{enumerate}
\item $0\To \Omega^{n}_{-} \xto{C^{-1}} H^{n}(\Omega^\blob_{-})\To 0\To 0$.
\item $0\To Z\Omega^n_{-}\To\Omega^n_{-}\To \Omega^n_{-}/Z\Omega^n_{-}\To 0$.
\item $0\To\Omega^n_{-}/Z\Omega^n_{-}\xTo{d}Z\Omega^{n+1}_{-}\To H^{n+1}(\Omega^\blob_{-})\To 0$.
\end{enumerate}
The first complex is exact (i.e., $C^{-1}$ is an isomorphism) on any $\bb F_p$-algebra satisfying (Sm3), while the second two complexes are obviously exact on any $\bb F_p$-algebra. We now use Lemma \ref{lemma}(i) repeatedly. First it follows that $H^n(\Omega^\blob_{-})$ is already derived at any Cartier smooth $\bb F_p$-algebra, whence the same is true of $Z\Omega^0_{-}=H^0(\Omega^\blob_{-})$. By induction, assuming that $Z\Omega^n_{-}$ is already derived at any Cartier smooth $\bb F_p$-algebra, the second sequence shows that the same is true of $\Omega^n_{-}/Z\Omega^n_{-}$, and then the third sequence shows that the same is true of $Z\Omega^{n+1}_{-}$. We have thus shown that $Z\Omega^n_{-}$ is already derived at any Cartier smooth $\bb F_p$-algebra, for all $n\ge0$.

The obvious cocycle-coboundary short exact sequence (and Lemma \ref{lemma}(i) again) then shows that $B\Omega^n_{-}$ is also already derived at any Cartier smooth $\bb F_p$-algebra, for all $n\ge0$.

Now recall that the inverse Cartier map is used to define the higher coboundaries and cocycles for any $\bb F_p$-algebra $A$, \[0=:B_0\Omega_{A}^n\subseteq B_1\Omega_{A}^n\subseteq\cdots \subseteq Z_1\Omega_{A}^n\subseteq Z_0\Omega_{A}^n:=\Omega_{A}^n,\] inductively as follows, for $r\ge1$:
\begin{itemize}\itemsep0pt
\item $Z_r\Omega_{A}^n$ is the subgroup of $\Omega_{A}^n$ containing $B\Omega^n_{A}$ and satisfying $Z_r\Omega_{A}^n/B\Omega_{A}^n=C^{-1}(Z_{r-1}\Omega_{A}^n)$.
\item $B_r\Omega_{A}^n$ is the subgroup of $\Omega_{A}^n$ containing $B\Omega^n_{A}$ and satisfying $B_r\Omega_{A}^n/B\Omega_{A}^n=C^{-1}(B_{r-1}\Omega_{A}^n)$.
\end{itemize}
(Minor notational warning: $B_1\Omega^n_{A}=B\Omega^n_{A}$, but the inclusion $Z_1\Omega^n_{A}\subseteq Z\Omega^n_{A}$ is strict if the inverse Cartier map is not surjective.) Iterating the inverse Cartier map defines maps \[C^{-r}:\Omega_{A}^n\To Z_r\Omega_{A}^n/B_r\Omega_{A}^n,\] which are (for purely formal reasons) isomorphisms for all $r,n\ge 0$ if $A$ satisfies (Sm3). By using Lemma~\ref{lemma}(i) in the same way as above, a simple induction on $r$ then shows that the functors $Z_r\Omega_{-}^n$ and $B_r\Omega_{-}^n$ are already derived at any Cartier smooth $\bb F_p$-algebra.

We now begin to analyse the de Rham--Witt groups, which are more subtle. We will use the following collection of functors $\bb F_p\op{-algs}\to\op{Ab}$ and maps between then:
\[\xymatrix@R=6mm{
&&&0&&\\
0\ar[r] & Z_{r-1}\Omega^{n-1}_{-}\ar[r]&\Omega^{n-1}_{-}\ar[r]^-{dV^{r-1}}&W_r\Omega^n_{-}/V^{r-1}\Omega^n_{-}\ar[u]\ar[r]^-{\res R}&W_{r-1}\Omega^n_{-}\ar[r] &0\\
&&&W_r\Omega^n_{-}\ar[u]&&\\
&&&\Omega^n_{-}\ar[u]_{V^{r-1}}&&\\
&&&B_{r-1}\Omega^n_{-}\ar[u]&&\\
&&&0\ar[u]&&
}\]
We break this diagram into sequences (not even necessarily complexes; for example, it is not clear if $Z_{r-1}\Omega^{n-1}_{A}$ is killed by $dV^{r-1}$ for arbitrary $A$) of functors as follows:
\begin{enumerate}
\item $0\to Z_{r-1}\Omega^{n-1}_{-}\to\Omega^{n-1}_{-}\to\ker \res R\to 0$.
\item $0\to \ker \res R\to W_r\Omega^n_{-}/V^{r-1}\Omega^n_{-}\xto{\res R}W_{r-1}\Omega^n_{-}\to 0$.
\item $0\to B_{r-1}\Omega^n_{-}\to\Omega^n_{-}\xto{V^{r-1}}V^{r-1}\Omega^n_{-}\to 0$
\item$0\to V^{r-1}\Omega^n_{-}\to W_r\Omega^n_{-}\to W_r\Omega^n_{-}/V^{r-1}\Omega^n_{-}\to 0$.
\end{enumerate}
The row and column in the above diagram (hence sequences (i)--(iv)) are exact upon evaluating on any smooth $\bb F_p$-algebra, by Illusie \cite[(3.8.1) \& (3.8.2)]{Illusie1979}. 

We now claim inductively on $r\ge 1$ that $W_r\Omega^n_A$ is supported in degree $0$ on any Cartier smooth $\bb F_p$-algebra. Note that we have already established stronger statements for $\Omega^n_-$, $Z_r\Omega^n_-$, $B_r\Omega^n_r$ for all $r,n\ge0$. So Lemma \ref{lemma}(ii) and sequence (i) shows that $\ker \res R$ is supported in degree $0$ on any Cartier smooth $\bb F_p$-algebra; the same argument using sequence (iii) shows that the same is true of $V^{r-1}\Omega^n_-$. Then Lemma \ref{lemma}(iii) and sequence (ii) (and the inductive hypothesis) show that  $W_r\Omega^n_{-}/V^{r-1}\Omega^n_{-}$ is supported in degree $0$ on any Cartier smooth $\bb F_p$-algebra; the same argument using sequence (iv) shows that the same is true of $W_r\Omega^-$, completing the inductive proof of the claim.

To finish the proof we will show that $W_r\Omega^n_{-}$ is right exact, i.e., $H_0(\bb LW_r\Omega^n_{A})\isoto W_r\Omega^n_{A}$ for all $\bb F_p$-algebras $A$. We may and do suppose that $A$ is of finite type over $\bb F_p$. The key is to observe that the data \[\bigoplus_{n\ge0}H_0(\bb LW_r\Omega^n_{A})\qquad r\ge1\] forms a Witt complex for $A$. This follows from functoriality of left Kan extension and the fact that $n=0$ part of the above sum is simply $W_r(A)$ (indeed, Lemma \ref{lemma}(i) and the short exact sequences $0\to W_{r-1}\to W_r\to W_1\to 0$ imply inductively that the functors $W_r$ are already derived at $A$). Therefore universality of the de Rham--Witt complex yields natural maps $\lambda_A:W_r\Omega^n_{A}\to H_0(\bb LW_r\Omega^n_{A})$ such that the composition back to $W_r\Omega^n_{A}$ is the identity. Picking a finitely generated free $\bb F_p$-algebra $P$ and surjection $P\onto A$, the naturality provides us with a commutative diagram
\[\xymatrix{
W_r\Omega^n_{P}\ar[d]\ar[r]^-{\lambda_P} &H_0(\bb LW_r\Omega^n_{P})\ar[d]\ar[r]^-\cong &W_r\Omega^n_{P}\ar[d]\\
W_r\Omega^n_{A}\ar[r]^-{\lambda_A} &H_0(\bb LW_r\Omega^n_{A})\ar[r] &W_r\Omega^n_{A}
}\]
The left and right vertical arrows are surjective since $P\to A$ is surjective; the middle vertical arrow is surjective because $H_0(\bb LW_r\Omega^n_{A})$ is a quotient of $W_r\Omega^n_P$ (by definition of $H_0$, since we may as well suppose that $P$ is the $0$-simplices of a simplicial resolution of $A$). The indicated isomorphism holds because $P$ is free. Since both horizontal compositions are the identity, it follows from a diagram chase that all the horizontal maps are isomorphisms. This completes the proof that $W_r\Omega^n_-$ is right exact.
\end{proof}

We next study the usual filtrations on the Rham--Witt complex. Recall that if $A$ is any $\bb F_p$-algebra, then the descending {\em canonical}, {\em $V$-}, and {\em $p$}-{\em filtrations} on $W_r\Omega_A^n$ are defined for $i\ge0$ respectively by
\begin{align*}
\op{Fil}^iW_r\Omega_A^n&:=\ker(W_r\Omega_A^n\xto{R^{r-i}}W_i\Omega_A^n),\\
\op{Fil}^i_VW_r\Omega_A^n&:=V^iW_{r-i}\Omega_A^n+dV^iW_{r-i}\Omega_A^{n-1},\\
\op{Fil}_p^iW_r\Omega_A^n&:=\ker(W_r\Omega_A^n\xto{p^{r-i}}W_r\Omega_A^n).
\end{align*}
Standard de Rham--Witt identities show that \[\op{Fil}^iW_r\Omega_A^n\supseteq \op{Fil}^i_VW_r\Omega_A^n\subseteq \op{Fil}^i_pW_r\Omega_A^n.\] It was proved by Illusie \cite[Prop.~I.3.2 \& I.3.4]{Illusie1979} that these three filtrations coincide if $A$ is a smooth $\bb F_p$-algebra. It was then shown by Hesselholt that the canonical and $V$-filtrations in fact coincide for any $\bb F_p$-algebra (the key observation is that for any fixed $i\ge1$, the groups $W_{r+i}\Omega_A^n/\op{Fil}_V^iW_{r+i}\Omega_A^n$, for $n\ge0$ and $r\ge1$, have an induced structure of a Witt complex over $A$; a detailed proof in the generality of log structures may be found in \cite[Lem.~3.2.4]{Hesselholt2003}). For a general $\bb F_p$-algebra, the $p$- and canonical filtrations need not coincide, but we show that this holds for Cartier smooth algebras:

\begin{proposition}\label{proposition_filtrations}
Let $A$ be a Cartier smooth $\bb F_p$-algebra, and $i,n\ge0$, $r\ge1$.
\begin{enumerate}
\item the functors $\op{Fil}^iW_r\Omega_-^n$, $\op{Fil}_V^iW_r\Omega_-^n$, $\op{Fil}^i_pW_r\Omega_-^n$, $p^iW_r\Omega^n_-$, $W_r\Omega^n_-/p^iW_r\Omega^n_-$ are already derived at~$A$;
\item $\op{Fil}^iW_r\Omega_A^n=\op{Fil}_V^iW_r\Omega_A^n=\op{Fil}^i_pW_r\Omega_A^n$;
\end{enumerate}
\end{proposition}
\begin{proof}
We begin by reducing all assertions to the inclusion \[\op{Fil}^iW_r\Omega_A^n\supseteq \op{Fil}^i_pW_r\Omega_A^n\tag{\dag},\] so assume for a moment that this has been established.

Note first (even without (\dag)) from Proposition \ref{proposition_WOmega_already_derived} and Lemma \ref{lemma}(i) that $\Fil^iW_{r}\Omega^n_-=\ker(R^{r-i}:W_r\Omega_-^n\onto W_i\Omega_-^n)$ is already derived at $A$. Moreover, as explained before the proposition, one has $\op{Fil}^iW_r\Omega_-^n=\op{Fil}^i_VW_r\Omega_-^n\subseteq \op{Fil}^i_pW_r\Omega_-^n$, with equality in the case of smooth $\bb F_p$-algebras, so left Kan extension and (\dag) give part (ii) and imply that $\op{Fil}^iW_r\Omega_-^n$ and $\op{Fil}^i_pW_r\Omega_-^n$ are already derived at $A$. Then applying Lemma \ref{lemma}(i) to the functors $\Fil^{r-i}_pW_r\Omega^n_-\to W_r\Omega^n_-\xto{p^i}p^iW_r\Omega^{n}_-$ (the first two terms of which have already been shown to be already derived at $A$) shows that $p^iW_r\Omega^{n}_-$ is already derived at $A$, and then applying the same argument to the sequence $p^iW_r\Omega^{n}_-\to W_r\Omega^n_-\to W_r\Omega^{n}_-/p^iW_r\Omega^{n}_-$ shows that $W_r\Omega^{n}_-/p^iW_r\Omega^{n}_-$ is also already derived at $A$.

It remains to prove (\dag). For any smooth $\bb F_p$-algebra $R$ there is a natural map of short exact sequences
\[\xymatrix{
0\ar[r] & \Omega^n_{R}/B_{r-1}\Omega^n_{R}\ar[r]^-{V^{r-1}}\ar[d]_{C^{-1}} & \op{Fil}^{r-1}W_{r}\Omega_R^n\ar[r]\ar[d]_{\ul p}\ar[r]^-{\beta}& \Omega^{n-1}_{R}/Z_{r-1}\Omega^{n-1}_{R}\ar[r]\ar[d]_{C^{-1}} & 0\\
0\ar[r] & \Omega^n_{R}/B_{r}\Omega^n_{R}\ar[r]^-{V^{r}} & \op{Fil}^{r}W_{r+1}\Omega_R^n\ar[r]^-{\beta}& \Omega^{n-1}_{R}/Z_{r}\Omega^{n-1}_{R}\ar[r] & 0
}\]
\cite[(3.10.4)]{Illusie1979} in which $\ul p$ is characterised by the property that the composition \[\op{Fil}^{r-1}W_{r+1}\Omega_A^n\xto{R} \op{Fil}^{r-1}W_{r}\Omega_A^n\xto{\ul p}\op{Fil}^{r}W_{r+1}\Omega_A^n\subseteq \op{Fil}^{r-1}W_{r+1}\Omega_A^n\] is multiplication by $p$. (The definition of $\beta$ is unimportant but given in the top row by $\beta(V^{r-1}x+dV^{r-1}y)=y$, for $x\in\Omega^{n}_A$ and $y\in\Omega^{n-1}_A$; for $\beta$ in the bottom row replace $r-1$ by $r$.)

Upon left Kan extending and evaluating at $A$, we obtain the same map of short exact sequences but with $R$ replaced by $A$: here we use that $\Omega^n_{-}$, $B_{r}\Omega^n_{-}$, $Z_{r}\Omega^n_{-}$, and $\Fil^iW_{r}\Omega^n_-$ are already known to be derived at $A$. But the inverse Cartier maps \[C^{-1}:\Omega^n_{A}/B_{r-1}\Omega^n_A\To \Omega^n_{A}/B_r\Omega^n_A\qquad C^{-1}:\Omega^{n-1}_{A}/Z_{r-1}\Omega^{n-1}_A\To \Omega^{n-1}_{A}/Z_r\Omega^{n-1}_A\] are injective: this follows formally from the injectivity of $C^{-1}:\Omega^n_A\to\Omega^n_A/B\Omega^n_A$ (similarly for $n-1$) and the surjectivity (by definition of the higher coboundaries and cocycles) of $C^{-1}:B_{r-1}\Omega^n_A\to B_{r}\Omega^n_A/B\Omega^n_A$ and $C^{-1}:Z_{r-1}\Omega^{n-1}_A\to Z_{r}\Omega^{n-1}_A/B\Omega^{n-1}_A$. Therefore the middle vertical arrow $\ul p:\op{Fil}^{r-1}W_{r}\Omega_A^n\to \op{Fil}^{r}W_{r+1}\Omega_A^n$ is also injective and so, from the characterisation of $\ul p$, we deduce that \[\ker(\op{Fil}^{r-1}W_{r+1}\Omega_A^n\xto{p} \op{Fil}^{r-1}W_{r+1}\Omega_A^n)\subseteq \ker(\op{Fil}^{r-1}W_{r+1}\Omega_A^n\xto{R} \op{Fil}^{r-1}W_{r}\Omega_A^n)\,\left(=\op{Fil}^{r}W_{r+1}\Omega^n_A\right).\] 

It now follows easily by induction that $\Fil^{r}_pW_{r+1}\Omega^n_A\subseteq \op{Fil}^{r}W_{r+1}\Omega^n_A$. Indeed, this is vacuous if $r=0$, while if $\omega\in\op{Fil}^r_pW_{r+1}\Omega^n_A$ then $R\omega\in \op{Fil}^{r-1}_pW_{r}\Omega^n_A\subseteq \op{Fil}^{r-1}W_{r}\Omega^n_A$ (the inclusion holds by the inductive hypothesis) and so $\omega$ belongs to $\op{Fil}^{r-1}W_{r+1}\Omega^n_A$ and is killed by $p$, whence the previous equation shows that $\omega\in\op{Fil}^rW_{r+1}\Omega^n_A$, as desired.

Replacing $r$ by $r-1$ to simplify notation, we have shown that $\Fil^{r-i}_pW_{r}\Omega^n_A\subseteq \op{Fil}^{r-i}W_{r}\Omega^n_A$ for $i=1$. This easily extends to all $i\ge1$ by induction (which completes the proof): indeed, if $\omega\in \Fil^{r-i}_pW_{r}\Omega^n_A$ then $p^{i-1}\in\Fil^{r-1}_pW_r\Omega^n_A$, whence $R(p^{i-1}\omega)=0$ by the case $i=0$; in other words, $R\omega\in\op{Fil}^{r-1-(i-1)}_pW_{r-1}\Omega^n_A$ and so the inductive hypothesis implies $R\omega\in\op{Fil}^{r-1-(i-1)}W_{r-1}\Omega^n_A$, i.e., $R^i\omega=0$, as required.
\end{proof}

\begin{corollary}\label{corollary_de_Rham_quis}
Let $A$ be a Cartier smooth $\bb F_p$-algebra. The canonical maps \[[0\to W_r\Omega^1_A/p\xto{d}\cdots \xto{d}W_r\Omega^{n-1}_A/p\xto{d}W_r\Omega^n_A/VW_{r-1}\Omega^n_A\xto{d}0]\To \Omega_A^{\le n}\] and \[W_r\Omega_A^\blob/p\to \Omega^\blob_A\] are quasi-isomorphisms.
\end{corollary}
\begin{proof}
The sequence of functors \[W_r\Omega^n_-\xto{V} W_{r+1}\Omega^n_-\xto{F^r}Z_r\Omega^n_-\To 0\] is exact when evaluated at any smooth $\bb F_p$-algebra \cite[(3.11.3)]{Illusie1979}. Since the three functors are already derived at $A$ by Proposition \ref{proposition_WOmega_already_derived}, left Kan extension shows that the sequence is also exact when evaluated at $A$. (Here we implicitly use right exactness of left Kan extension: if $\cal F\to \cal G$ is a map of functors $\bb F_p\op{-algs}\to\op{Ab}$ such that $\cal F(R)\to \cal G(R)$ is surjective for all free $R$, then $H_0(\bb L\cal F(B))\to H_0(\bb L\cal G(B))$ is surjective for all $B\in\bb F_p\op{-algs}$.)

It follows that the functor $W_r\Omega^n_-/VW_{r-1}\Omega^n_-$ is already derived at $A$. Since the functors $W_r\Omega^i_-/p$ and $\Omega^i_-$ are also already derived at $A$ (by Proposition \ref{proposition_filtrations}(i) and Proposition \ref{proposition_WOmega_already_derived} respectively), the quasi-isomorphisms follow by left Kan extension from the case of finitely generated free $\bb F_p$-algebras, namely \cite[Corol.~3.20]{Illusie1979}.
\end{proof}

\begin{theorem}\label{theorem_eta}
Let $A$ be a Cartier smooth $\bb F_p$-algebra. Then
\begin{enumerate}
\item $W\Omega_A^n:=\projlim_rW_r\Omega_A^n$ is $p$-torsion-free.
\item $F^iW\Omega^n_A=d^{-1}(p^iW\Omega^n_A)$ for all $i\ge1$.
\end{enumerate}
\end{theorem}
\begin{proof}
Part (i) follows formally from the equality of the canonical and $p$-filtrations in Proposition \ref{proposition_filtrations}(ii).

For (ii), let us first fix a finite level $r$ and show that an element $x\in W_r\Omega^n_A$ is in the image of $F:W_{r+1}\Omega^n_A\to W_r\Omega^n_A$ if and only if $dx\in pW_r\Omega^{n+1}_A$. The implication ``only if'' following from the identity $dF=pFd$, so we prove ``if'' by induction on $r$; assume the result known at level $r-1$ and suppose $dx\in pW_r\Omega^{n+1}_A$. Then $dR^{r-1}\tilde x=0$, whence $\tilde x\in FW_2\Omega^n_A$ (by the $r=1$ case of the right exact sequence from the start of proof of Corollary \ref{corollary_de_Rham_quis}; this also proves the base case of the induction). Lifting to level $r$ and recalling that $\Fil^1W_{r}\Omega^n_A=\Fil^1_VW_{1}\Omega^n_A$ allows us to write $x=Fz+Vz_1+dVz_2$ for some $z\in W_{r+1}\Omega^n_A$, $z_1\in W_{r-1}\Omega^n_A$, $z_2\in W_{r-1}\Omega^{n-1}_A$.

Applying $Fd$ to the previous identity yields $dz_1=FdVz_1=Fdx-pF^2dz\in pW_{r-1}\Omega^{n+1}_A$, whence the inductive hypothesis allows us to write $z_1=Fz_3$ for some $z_3\in W_r\Omega^n_A$. In conclusion, \[x=Fz+VFz_1+dVz_2=Fz+FVz_3+FdV^2z_2\in FW_{r+1}\Omega^n_A,\] as required to complete the induction.

We now prove (ii); it suffices to treat the case $i=1$ as the general case then follows by induction using part (i). The identity $VF=p$ shows that the kernel of $F:W_{r+1}\Omega^n_A\to W_{r}\Omega^n_A$ is killed by $p$, hence is killed by $R$ by Proposition \ref{proposition_filtrations}(ii). So, using also the finite level case of the previous paragraphs, the sequence of pro abelian groups \[0\To\{W_r\Omega^n_A\}_r\xTo{F} \{W_r\Omega^n_A\}_r\xTo{d} \{W_r\Omega^{n+1}_A/pW_r\Omega^{n+1}_A\}_r\] is exact, and taking the limit yields the short exact sequence \[0\To W\Omega^n_A\xTo{F}W\Omega^n_A\xTo{d}\projlim_r W_r\Omega^{n+1}_A/pW_r\Omega^{n+1}_A.\tag{\dag}\] But the sequence of pro abelian groups \[0\To\{W_r\Omega^{n+1}_A\}_r\xTo{p} \{W_r\Omega^{n+1}_A\}_r\To \{W_r\Omega^{n+1}_A/pW_r\Omega^{n+1}_A\}_r\] is also exact (again since $p$-torsion is killed by $R$), and so taking the limit shows that the final term in (\dag) is $W\Omega^{n+1}_A/pW\Omega^{n+1}_A$, as required to complete the proof.
\end{proof}

\begin{remark}
In terms of the functor $\eta_p$ \cite[\S6]{Bhatt_Morrow_Scholze2}, the previous theorem implies that $\phi:W\Omega^\blob_A\to\eta_pW\Omega^\blob_A$ is an isomorphism and hence provides a quasi-isomorphism $\phi:W\Omega^\blob_A\quis L\eta_pW\Omega^\blob_A$.
\end{remark}

\subsection{The dlog forms of a Cartier smooth algebra}\label{ss_dlog}
Recall from Remark \ref{remark_dlog} that, for any $\bb F_p$-algebra $A$, we have defined the subgroup of dlog forms $W_r\Omega^n_{A,\sub{log}}\subseteq W_r\Omega^n_{A}$. Alternatively, standard de Rham--Witt identifies imply the existence of a unique map $\res F:W_r\Omega_A^n\to W_r\Omega_A^n/dV^{r-1}\Omega_A^{n-1}$ making the following diagram commute (in which $\pi$ denotes the canonical quotient map)
\[\xymatrix@C=1.5cm{
W_{r+1}\Omega_A^n\ar[d]_R\ar[r]^F & W_r\Omega_A^n\ar[d]^\pi\\
W_r\Omega_A^n\ar[r]^{\res F} & W_r\Omega_A^n/dV^{r-1}\Omega_A^{n-1}
}\]
and it may then be shown that \[W_r\Omega_{A,\sub{log}}^n= \ker(W_r\Omega_A^n\xto{\res F-\pi}W_r\Omega_A^n/dV^{r-1}\Omega_A^{n-1}).\] See \cite[Corol.~4.2]{Morrow_pro_GL} for further details.

The following argument is modelled on the proof in special case when $r=1$ and $A$ is smooth \cite[Prop.~2.26]{ClausenMathewMorrow}:

\begin{lemma}\label{lemma_dlog}
Let $A$ be a Cartier smooth $\bb F_p$-algebra. Then the square
\[\xymatrix{
W\Omega^n_{A}/p^r\ar[r]^{F-1}\ar[d] & W\Omega^n_{A}/p^r\ar[d]\\
W_r\Omega^n_{A}\ar[r]^-{\res F-\pi} & W_r\Omega^n_{A}/dV^{r-1}\Omega^n_{A}
}\]
is bicartesian.
\end{lemma}
\begin{proof}
Since the vertical arrows are surjective, we must show that their kernels are isomorphic, i.e, that \[F-1:\op{Im}(V^rW\Omega^n_A+dV^rW\Omega^n_A\to W\Omega^n_A/p^r)\Isoto \op{Im}(V^rW\Omega^n_A+dV^{r-1}\Omega^n\to W\Omega^n_A/p^r).\] Surjectivity follows from surjectivity modulo $p$, which is a consequence of the identities in $W\Omega^n_A$ \[(F-1)V\equiv -V\mod{p},\qquad (F-1)d\sum_{i>0}V^i=d.\] It remains to establish injectivity, so suppose that $x\in W\Omega^n_A$ is an element of the form $x=V^ry+dV^rz$ such that $(F-1)x\in p^rW\Omega^n_A$. Applying $1+\cdots+F^{r-1}$ we deduce that $(F^r-1)x\in p^rW\Omega^n_A$, whence $dV^ry=dx=p^rF^rdx-d(F^r-1)x\in p^rW\Omega^n_A$ and so $dy=F^rdV^ry\in p^rW\Omega^n_A$; therefore $y=F^ry'$ for some $y'\in W\Omega^n_A$ by Theorem \ref{theorem_eta}(ii).

We know now that $x=p^ry'+dV^rz$; applying $F^r-1$ shows that $(F^r-1)dV^rz=d(1-V^r)z$ is divisible by $p^r$, whence $(1-V^r)z\in F^rW\Omega^n_A$ by Theorem \ref{theorem_eta}(ii) again. But $1-V^r$ is an automorphism of $W\Omega^n_A$ with inverse $\sum_{i\ge0}V^{ri}$ (which commutes with $F$), whence $z=F^rz'$ for some $z'\in W\Omega^n_A$. In conclusion $x=p^ry'+p^rdz'$, which completes the proof of injectivity.
\end{proof}

In the case of smooth $\bb F_p$-algebras, the following is an important result due to Illusie \cite{Illusie1979}:

\begin{theorem}
Let $A$ be a Cartier smooth, local $\bb F_p$-algebra. Then the canonical map $W_{s}\Omega_{A,\sub{log}}^n/p^r\to W_r\Omega_{A,\sub{log}}^n$ is an isomorphism for any $s\ge r\ge1$.
\end{theorem}
\begin{proof}
The sequence
\[0\To W\Omega^n_A/p^{s-r}\xTo{p^r} W\Omega^n_A/p^s\To W\Omega^n_A/p^r\To 0\]
is short exact since $W\Omega^n_A$ is $p$-torsion-free by Theorem \ref{theorem_eta}. Taking $F$-fixed points and appealing to Lemma \ref{lemma_dlog} gives an exact sequence \[0\To W_{s-r}\Omega^n_{A,\sub{log}}\xTo{p^r}W_{s}\Omega_{A,\sub{log}}^n\to W_r\Omega_{A,\sub{log}}^n.\] But $W_s\Omega^n_{A,\sub{log}}\to W_r\Omega^n_{A,\sub{log}}$ is surjective since both sides are generated by dlog forms, and the analogous surjectivity for $r$ replaced by $r-s$ shows that the image of the $p^r$ arrow in the previous line is $p^rW_s\Omega^n_{A,\sub{log}}$.
\end{proof}

\subsection{Proof of Theorem \ref{theorem_GL_val_general} via the trace map}
This subsection is devoted to the proof of Theorem \ref{theorem_GL_val_general}. Familiarity with topological cyclic homology is expected (we use the ``old'' approach to topological cyclic homology in terms of the fixed points spectra $\TR^r(-;p)=\THH(-)^{C_{p^{r-1}}}$, though the argument could be reformulated in terms of Nikolaus--Scholze's approach \cite{NikolausScholze2017} and the motivic filtrations of \cite{Bhatt_Morrow_Scholze3}.)

Given any smooth $\bb F_p$-algebra $R$, the homotopy groups of $\TR^r(R;p)$ have been calculated by Hesselholt \cite[Thm.~B]{Hesselholt1996}: \[\TR^r_n(R;p)\cong\bigoplus_{i\ge0} W_r\Omega_{R}^{n-2i}.\] These isomorphisms are natural in $R$ and compatible with $r$ in the sense that the restriction map $R^r:\TR^{2r}(R;p)\to \TR^r(R;p)$ induces zero on $W_{2r}\Omega_{R}^{n-2i}$, for $i>0$, and the usual de Rham--Witt restriction map $R^r:W_{2r}\Omega_{R}^{n}\to W_r\Omega_{R}^{n}$, for $i=0$.

Therefore left Kan extending\footnote{Note that $\TR^r(-;p):\bb F_p\op{-algs}\to\op{Sp}$ is left Kan extended from finitely generated free $\bb F_p$-algebras. Indeed, this is true for $\THH$ since tensor products and geometric realisations commute with sifted colimits, and then it is true by induction for $\TR^r(-;p)$ thanks to the isotropy separation sequences and the fact that homotopy orbits commute with all colimits. See also \cite[Corol.~6.8]{Bhatt_Morrow_Scholze3} for the construction of this filtration in the case of $\THH$.} shows that on any $\bb F_p$-algebra $A$ there is a natural descending $\bb N$-indexed filtration on $\TR^r(A;p)$ with graded pieces $\bigoplus_{i\ge0} \bb LW_r\Omega_{A}^{n-2i}$; moreover, the restriction map $R^r:\TR^{2r}(A;p)\to \TR^r(A;p)$ is compatible with this filtration and behaves analogously to the previous smooth case.

Now let $A$ be any Cartier smooth $\bb F_p$-algebra. Then $\bb LW_r\Omega_{A}^{n-2i}\simeq W_r\Omega_{A}^{n-2i}[0]$ by Proposition \ref{proposition_WOmega_already_derived} and so the previous paragraph implies that the description of $\TR_n^r(-;p)$ from the smooth case remains valid for $A$. In particular we obtain natural isomorphisms of pro abelian groups $\{\TR_n^r(A;p)\}_{r\sub{ wrt }R}\cong\{W_r\Omega^n_{A}\}_{r\sub{ wrt }R}$, whence taking the limit calculates the homotopy groups of $\TR(A;p):=\op{holim}_r\TR^r(A;p)$ as \[\TR_n(A;p)\cong W\Omega^n_A.\] Writing $\TR(A;\bb Z/p^r\bb Z)=\TR(A;p)/p^r$, we apply Theorem \ref{theorem_eta}(i) to obtain \[\TR(A;\bb Z/p^r\bb Z)\cong W\Omega^n_{A}/p^r.\]

The previous two isomorphisms are compatible with the Frobenius endormorphisms $F$ of $\TR(R;\bb Z_p)$ and $W\Omega_{A}$ (by left Kan extending the Frobenius compatibility of \cite[Thm.~B]{Hesselholt1996} from the smooth case), thereby proving that the homotopy groups of $\TC(A;\bb Z/p^r):=\op{hofib}(F-1:\TR(A;\bb Z/p^r\bb Z)\to \TR(A;\bb Z/p^r\bb Z))$ fit into a long exact homotopy sequence \[\cdots \To \TC_n(A;\bb Z/p^r\bb Z)\To W\Omega^n_{A}/p^r\xto{F-1} W\Omega^n_{A}/p^r\To\cdots\] In this long exact sequence we may replace each map $F-1$ by $\res F-\pi:W_r\Omega^n_{A}\to W_r\Omega^n_{A}/dV^{r-1}\Omega^n_{A}$, by Lemma \ref{lemma_dlog}, and so obtain a short exact sequence
\[0\To \op{Coker}(W_r\Omega^{n+1}_{A}\xto{\res F-\pi}W_r\Omega^{n+1}_{A}/dV^{r-1}\Omega^n_A)\To \TC_n(A;\bb Z/p^r\bb Z)\To W_r\Omega^n_{A,\sub{log}}\To 0\]
(here we use the equality $W_r\Omega^n_{A,\sub{log}}=\ker(\res F-\pi)$ from the start of \S\ref{ss_dlog}).

Now pick a free $\bb F_p$-algebra surjecting onto $A$ and let $R$ denote its Henselisation along the kernel; then $R$ is a local ind-smooth $\bb F_p$-algebra and so the previous short exact sequence is also valid for $R$ (since $R$ is also Cartier smooth). Putting together the short exact sequences for $A$ and $R$ yields a commutative diagram with exact rows, in which we have included the trace maps from algebraic $K$-theory:
\[\xymatrix@R=3mm@C=6mm{
K_n(R;\bb Z/p^r\bb Z)\ar@/^3mm/[drrr]^{\op{tr}}\ar[dd]&&&&\\
&0\ar[r] &\op{Coker}(W_r\Omega^{n+1}_{R}\xto{\res F-\pi}W_r\Omega^{n+1}_{R}/dV^{r-1}\Omega^n_R)\ar[dd] \ar[r] & \TC_n(R;\bb Z/p^r\bb Z)\ar[r]\ar[dd] & W_r\Omega^n_{R,\sub{log}}\ar[dd]\ar[r]&0\\
K_n(A;\bb Z/p^r\bb Z)\ar@/^3mm/[drrr]^{\op{tr}}&&&&\\
&0\ar[r] &\op{Coker}(W_r\Omega^{n+1}_{A}\xto{\res F-\pi}W_r\Omega^{n+1}_{A}/dV^{r-1}\Omega^n_A) \ar[r] & \TC_n(A;\bb Z/p^r\bb Z)\ar[r] & W_r\Omega^n_{A,\sub{log}}\ar[r]&0
}\]

We claim that the composition \[c_A:K_n(A;\bb Z/p^r\bb Z)\xto{\op{tr}}\TC_n(A;\bb Z/p^r\bb Z)\To W_r\Omega^n_{A,\sub{log}}\]  is an isomorphism. Since $R$ is ind-smooth and local, the analogous composition $c_R$ for $R$ in place of $A$ is an isomorphism by Geisser--Levine \cite{GeisserLevine2000} and Geisser--Hesselholt \cite{GeisserHesselholt1999}; see \cite[Thm.~2.48]{ClausenMathewMorrow} for a reminder of the argument in the case $r=1$. 

According to the main rigidity result of \cite{ClausenMathewMorrow}, the square of spectra formed by applying $\op{tr}:K(-;\bb Z/p^r\bb Z)\to \TC(-;\bb Z/p^r\bb Z)$ to $R\to A$ is homotopy cartesian, yielding a long exact sequence of homotopy groups
\[\cdots\To K_n(R;\bb Z/p^r\bb Z)\to K_n(A;\bb Z/p^r\bb Z)\oplus \TC_n(R;\bb Z/p^r\bb Z)\To \TC_n(A;\bb Z/p^r\bb Z)\To\cdots.\] But the trace map $\op{tr}:K_n(R;\bb Z/p^r\bb Z)\to \TC_n(R;\bb Z/p^r\bb Z)$ is injective (as $c_R$ is injective), whence this breaks into short exact sequences and proves that the trace map square in the above diagram is bicartesian.

Next, the vertical arrow between the cokernel terms is an isomorphism since $R\to A$ is Henselian along its kernel: see the proof of \cite[Prop.~6.12]{ClausenMathewMorrow}. Therefore the right-most square in the diagram is bicartesian. Concatanating the two bicartesian squares shows that
\[\xymatrix{
K_n(R;\bb Z/p^r\bb Z)\ar[r]^-{c_R}\ar[d] & W_r\Omega^n_{R,\sub{log}}\ar[d] \\ 
K_n(A;\bb Z/p^r\bb Z)\ar[r]^-{c_A}          & W_r\Omega^n_{A,\sub{log}}
}\]
is bicartesian. But we have already explained that $c_R$ is an isomorphism, whence $c_A$ is also an isomorphism. Finally, note that $c_A$ sends any symbol in $K_n(A;\bb Z/p^r\bb Z)$ to the associated dlog form in $W_r\Omega^n_{A,\sub{log}}$, by the trace map's multiplicativity  \cite[\S6]{GeisserHesselholt1999} and behaviour on units \cite[Lem.~4.2.3]{GeisserHesselholt1999}. The completes the proof of Theorem \ref{theorem_GL_val_general}. \qed

\section{$K$-theory of valuation rings}
\subsection{Gersten injectivity}\label{subsection_gersten}
Here we prove Theorem \ref{theorem_gersten_intro} from the introduction:

\begin{theorem}\label{proposition_gersten}
Let $\roi$ be a valuation ring containing a field, and $n\ge0$. Then $K_n(\roi)\to K_n(\Frac\roi)$ is injective.
\end{theorem}
\begin{proof}
Let $p\ge1$ be the characteristic exponent of $\Frac\roi$. It is enough to show that $K_n(\roi)[\tfrac1p]\to K_n(\Frac\roi)[\tfrac1p]$ is injective: indeed, if $p=1$ (i.e., $\Char\Frac\roi=0$) then this is exactly the desired assertion, while if $p>1$ then we may remove the $[\tfrac1p]$ since $K_n(\roi)$ has been shown to be $p$-torsion-free (Theorem \ref{theorem_GL_intro}).

By writing $F=\Frac \cal O$ as the union of its finitely generated subfields, and intersecting $\roi$ with each of these subfields, we may immediately reduce to the case that $F$ is finitely generated over its prime subfield $\bb F$. Fix a prime number $\ell\neq p$ and $\al\in\ker(K_n(\roi)\to K_n(F))$. Pick a finitely generated $\bb F$-subalgebra $A\subseteq \roi$ such that $\Frac A=F$ and such that $\al$ comes from a class $\al_A\in \ker(K_n(A)\to K_n(F))$. According to Gabber's refinement of de Jong's theory of alterations \cite[Thm.~2.1 of Exp.~X]{Gabber2014}, there exist a regular, connected scheme $X$ and a projective, generically finite morphism $X\to\Spec A$ such that $|K(X):F|$ is not divisible by $\ell$.

Let $\tilde{\roi}$ be the integral closure of $\roi$ in $K(X)$. The theory of valuation rings states that $\tilde\roi$ has only finitely many maximal ideal $\frak m_1,\dots,\frak m_d$ and that each localisation $\tilde\roi_{\frak m_i}$ is a valuation ring (indeed, these localisations classify the finitely many extensions of the valuation ring $\roi$ to $K(X)$ \cite[6.2.2]{GabberRamero2003}). By the valuative criterion for properness, each commutative diagram
\[\xymatrix{
\Spec K(X)\ar[d]\ar[rr] && X\ar[d]\\
\Spec \tilde\roi_{\frak m_i}\ar[r]\ar@{-->}[rru]^{\exists\rho_i}&\Spec \tilde\roi\ar[r]&\Spec A
}\]
may be filled in by a dashed arrow as indicated. Let $x_i\in X$ be the image of the maximal ideal of $\tilde{\roi}_{\frak m_i}$ under $\rho_i$.

Since $X$ is quasi-projective over $\bb F$, the homogeneous prime avoidance lemma implies the existence of an affine open of $X$ containing the finitely many points $x_1,\dots,x_d$. This is given by a regular sub-$A$-algebra $B\subseteq K(X)$ contained inside $\bigcap_i\tilde\roi_{\frak m_i}=\tilde\roi$ and such that $\Frac B=K(X)$. After replacing $B$ by its semi-localisation at the radical ideal $\bigcap_{i}\frak m_i\cap B$, we may suppose in addition that $B$ is semi-local. The situation may be summarised by the following diagram:
\[\xymatrix{
\Spec K(X)\ar[d]\ar[r] &\Spec B\ar[r]& X\ar[d]\\
\Spec \tilde\roi_{\frak m_i}\ar[r]&\Spec \tilde\roi\ar[r]\ar[u]&\Spec A
}\]

The element $\al_A$ vanishes in $K_n(K(X))$, hence also in $K_n(B)$ since Gersten's conjecture states that $K_n(B)\to K_n(K(X))$ is injective \cite[Thm.~7.5.11]{Quillen1973a} (note that $B$ is a regular, semi-local ring, essentially of finite type over a field). Therefore $\al$ vanishes in $K_n(\tilde\roi)$. Since $\tilde\roi$ is an integral extension of $\roi$, it is the filtered union of its $\roi$-algebras which are finite $\roi$-modules; let $\roi'\subseteq \tilde\roi$ be such a sub-$\roi$-algebra such that $\al$ vanishes in $K_n(\roi')$. Since finite torsion-free $\roi$-modules are finite free, there is a trace map $K_n(\roi')\to K_n(\roi)$ such that the composition $K_n(\roi)\to K_n(\roi')\to K_n(\roi)$ is multiplication by $m:=|\Frac\roi':F|$, which is not divisible by $\ell$. In conclusion, we have shown that $m\al=0$ for some integer $m$ not divisible by $\ell$. 

In other words, $K_n(\roi)_{(\ell)}\to K_n(F)_{(\ell)}$ is injective for all primes $\ell\neq p$, which impies the injectivity of $K_n(\roi)[\tfrac1p]\to K_n(F)[\tfrac1p]$ and so completes the proof.\qedhere
\end{proof}

\begin{remark}
The following streamlined proof of Theorem \ref{proposition_gersten} has been pointed out to us by Antieau. In case $\roi$ contains $\bb Q$, then resolution of singularities implies that $\roi$ is a filtered colimit of essentially smooth local $\bb Q$-algebras $A$, whence the desired injectivity immediately reduces to the Gersten conjecture for each $A$. When $\roi$ contains $\bb F_p$, we combine the $p$-torsion-freeness of  Theorem \ref{theorem_GL_intro} with the isomorphism of Remark \ref{remark_perf} to reduce to the case that $\roi$ is perfect; but then $\roi$ is a filtered colimit of essentially smooth local $\bb F_p$-algebras by Temkin's inseparable local uniformisation theorem \cite[Thm.~1.3.2]{Temkin2013}, which as in characteristic $0$ reduces the problem to the usual Gersten conjecture.
\end{remark}

\subsection{Homotopy invariance}
Now we prove Theorem \ref{theorem_K_of_val_intro} from the introduction, entirely using classical methods; this is independent of Section \ref{subsection_gersten}. We remark that parts of the following theorem have also recently been noticed by Antieau and Mathew.

\begin{theorem}\label{theorem_K_of_val}
Let $\roi$ be a valuation ring. Then
\begin{enumerate}
\item $K_n(\roi)\to K_n(\roi[T_1,\dots,T_d])$ is an isomorphism for all $d\ge1$ and $n\in\bb Z$;
\item $K_n(\roi)\to \KH_n(\roi)$ is an isomorphism for all $n\in\bb Z$;
\item $K_n(\roi)=0$ for all $n<0$.
\end{enumerate}
\end{theorem}

\comment{
\begin{remark}
Let $\roi$ be a valuation ring and set $F:=\Frac \roi$. Recall that the associated value group is the totally ordered abelian group $\Gamma:=F^\times/\roi^\times$ (written multiplicatively), ordered by the rule $f\le g\Leftrightarrow fg^{-1}\in\roi$. An subgroup $H\subseteq\Gamma$ is called {\em convex} (or sometimes {\em isolated}) if and only if, whenever $\al\le\beta\le\gamma$ are elements of $\Gamma$ such that $\al,\gamma\in H$ then also $\beta\in H$; the convex subgroups of $\Gamma$ are easily seen to be totally ordered by inclusion, and the {\em rank} of $\roi$ is defined to be the number of proper convex subgroups of $\Gamma$. In fact, the convex subgroups of $\Gamma$ are in one-to-one correspondence with the prime ideals of $\roi$ (which are also totally ordered by inclusion); this shows that the rank of $\roi$ is the same as its Krull dimension.
\end{remark}
}

\begin{proof}
Since assertions (i)--(iii) are compatible with filtered colimits we may assume that $\Frac\roi$ is finitely generated over its prime subfield, and in particular that $\roi$ is countable.

The usual proofs of the assertions (i)---(iii), in the case of a regular Noetherian ring, continue to work for $\roi$ once it is known, for all $d\ge0$, that
\begin{enumerate}[(a)]
\item $\roi[T_1,\dots,T_d]$ is coherent, and
\item every finitely presented $\roi[T_1,\dots,T_d]$-module has finite projective dimension.
\end{enumerate}
Indeed, see \cite[Eg.~1.4]{Weibel1989a} and \cite{Gersten1974} for a discussion of Quillen's fundamental theorem and applications in this degree of generality. See also \cite[Thm.~3.33]{AntieauGepnerHeller2016} for a proof of the vanishing of negative $K$-groups of regular, stably coherent rings.

It remains to verify that $\roi$ has properties (a) and (b). For (a), the coherence of polynomial algebras over valuation rings is relatively well-known: any finitely generated ideal $I$ of $\roi[T_1,\dots,T_d]$ is a flat $\roi$-module (since torsion-free modules over a valuation ring are flat), hence is of finite presentation over $\roi[T_1,\dots,T_d]$ by Raynaud--Gruson \cite[Thm.~3.6.4]{RaynaudGruson1971}. A textbook reference is \cite[Thm.~7.3.3]{Glaz1989}.

Next we check (b). A classical result of Osofsky \cite[Thm.~A]{Osofsky1967} \cite{Osofsky1974} states that a valuation ring $\roi$ has finite global dimension $n+1$ if and only if $\aleph_n$ is the smallest cardinal such that every ideal of $\roi$ can be generated by $\aleph_n$ elements. In particular, $\roi$ has finite global dimension $\le 2$ if and only if every ideal can be generated by at most countably many elements (i.e., $\roi$ is {\em $\aleph_0$-Noetherian}); but this is manifestly true in our case since we reduced to the case in which $\roi$ is itself a countable set. In conclusion $\roi$ has finite global dimension, which is inherited by any polynomial algebra over it by Hilbert \cite[Thm.~4.3.7]{Weibel1994}, and this is stronger than required in (b).
\end{proof}

\begin{remark}\label{remark_KH}
Let $X$ be a Noetherian scheme of finite Krull dimension. A conservative family of points for the cdh site $X_\sub{cdh}$ is given by the spectra of Henselian valuation rings \cite{GabberKelly2015}, and hence Theorem \ref{theorem_K_of_val}(iii) implies that $\cal K_n^\sub{cdh}=0$ for $n<0$, where $\cal K_n^\sub{cdh}$ denotes the sheafification of the abelian presheaf $\cal K_n(-)$ on $X_\sub{cdh}$. In fact, this follows from the earlier result of Goodwillie--Lichtenbaum \cite{GoodwillieLichtenbaum2001} that a conservative family of points for the rh site is given by the spectra of valuation rings.

Since $X_\sub{cdh}$ is known to have cohomological dimension $=\dim X$ \cite{SuslinVoevodsky2000}, we immediately deduce from the resulting descent spectral sequence that $K_n^\sub{cdh}(X)=0$ for $n<-\dim X$, where $K^\sub{cdh}(X):=\bb H(X_\sub{cdh},K)$ refers to the cdh sheafification of the $K$-theory presheaf of spectra. Moreover, combining Theorem \ref{theorem_K_of_val}(ii) with Cisinski's result \cite{Cisinski2013} (proved around 2010) that $\KH$ satisfies cdh descent implies that $K^\sub{cdh}(X)\simeq\KH(X)$, whence one immediately gets the following consequences:
\begin{enumerate}[(a)]
\item $\KH_n(X)=0$ for $n<-\dim X$; this was first proved in 2016 by Kerz--Strunk \cite{KerzStrunk2017} via their usage of Raynaud--Gruson flattening methods. (It is important to note that the proof of Theorem \ref{theorem_K_of_val} crucially uses Raynaud--Gruson to check that valuation rings are stably coherent.)
\item If $p$ is nilpotent in $X$ then $K_n(X)[\tfrac1p]=0$ for $n<-\dim X$ (since the nilpotence of $p$ implies $\KH(X)[\tfrac1p]=K(X)[\tfrac1p]$ and then we can apply (a); alternatively, once can avoid Cisinski's result by appealing instead to Theorem \ref{theorem_cdh_descent} to see that $K[\tfrac1p]$ satisfies cdh descent); this was first proved in 2011 by the first author's use of alterations and $\ell$dh topologies \cite{Kelly2014}.
\end{enumerate}
\end{remark}

\subsection{A surjectivity result}
Here we establish a surjectivity result (Corollary \ref{corollary_surj}) which was not stated in the introduction but which plays a role in the first author's new approach to comparing cdh and $\ell$dh topologies \cite[Thm.~1]{Kelly2018}.

\begin{lemma}\label{lemma_axiomatic_surj}
Let $k$ be a perfect field, $\roi\supseteq k$ a valuation ring, and $\frak p\subseteq\roi$ an ideal along which $\roi$ is Henselian.
\begin{enumerate}
\item If $F:k\op{-algs}\to\op{Sets}$ is a functor which commutes with filtered colimits and $\frak p$ is the maximal ideal of $\roi$, then $F(\roi)\to F(\roi/\frak p)$ is surjective.
\item If $F:k\op{-algs}\to\op{Spectra}$ is a functor which commutes with filtered colimits, is nil-invariant, and satisfies excision,\footnote{By ``excision'' we mean that whenever $A\to B$ is a homomorphism of $k$-algebras and $I\subseteq A$ is an ideal which is sent isomorphically to an ideal of $B$, then the induced map of relative theories $\op{hofib}(F(A)\to F(A/I))\to\op{hofib}(F(B)\to F(B/I))$ is an equivalence.} then $\pi_nF(\roi)\to \pi_nF(\roi/\frak p)$ is surjective for all $n\in\bb Z$.
\end{enumerate}
\end{lemma}
\begin{proof}
(i). Assume $\frak p=\frak m$ is the maximal ideal of $\roi$, and let $\al\in F(\roi/\frak m)$. Since $\roi/\frak m$ is the filtered union of $A/(\frak m\cap A)$, as $A$ runs over all finitely generated $k$-subalgebras of $\roi$, we pick such an $A$ such that $\al$ comes from some class $\al_{A/\frak m\cap A}\in F(A/(\frak m\cap A))$. Let $\frak q:=\frak m\cap A$ and let $A_\frak q^h$ be the Henselisation of $A_\frak q\subseteq\roi$ at $\frak q A_\frak q$; note that the inclusion $A\into\roi$ factors through $A_\frak q^h$, by functoriality of Henselisation. 

The completion $\hat{A_\frak q}$ admits a coefficient field containing $k$ by Cohen \cite[Thm.~60]{Matsumura1980}, whence the image of $\al_{A/\frak q}$ in $F(A_\frak q/\frak qA_\frak q)$ may be lifted to $K_n(\hat A_{\frak q})$; by Artin approximation (or N\'eron--Popescu desingularision) it may therefore be lifted to $\al_{A_\frak q^h}\in F(A_\frak q^h)$ (note here that $A_\frak q$ is essentially of finite type over a field, in particular it is excellent). The image of $\al_{A_\frak q^h}$ in $F(\roi)$ is of course the desired lift of $\al$.

(ii): Replacing $\frak p$ by its radical (using invariance of $F$ for locally nilpotent ideals), we may assume that $\frak p$ is a prime ideal. Then standard theory of valuation rings states that the localisation $\roi_\frak p$ is a valuation ring with maximal ideal $\frak p$. Thus we may apply excision to the data $\frak p\subseteq \roi\subseteq\roi_\frak p$ and so obtain a long exact Mayer--Vietoris sequence
\[ \dots \To \pi_nF(\OO) \To\pi_nF(\OO/\p) \oplus \pi_nF(\OO_\p) \To \pi_nF(\roi_\frak p/\frak p) \To \dots. \]
But $\roi_\frak p$ is Henselian along $\frak p$ (here we use the fact the Henselianness along an ideal depends only the ideal as a non-commutative ring, not on the ambient ring \cite[Corol.~1]{Gabber1992}), so part (i) implies that $\pi F_n(\roi_\frak p)\to \pi_nF(\roi_\frak p/\frak p)$ is surjective for all $n\in\bb Z$. Therefore the long exact sequence breaks into short exact sequences and moreover each map $\pi_nF(\roi)\to\pi_nF(\roi/\frak p)$ is surjective.
\end{proof}

\begin{corollary}\label{corollary_surj}
Let $\roi$ be a valuation ring containing a field, and $\frak p\subseteq\roi$ a prime ideal along which $\roi$ is Henselian. Then $K_n(\roi)\to K_n(\roi/\frak p)$ is surjective for all $n\in\bb Z$.
\end{corollary}
\begin{proof}
We apply Lemma \ref{lemma_axiomatic_surj}(ii) to the functor $\KH$, which satisfies all the hypotheses by Weibel \cite{Weibel1989a}, and then use Theorem \ref{theorem_K_of_val}(ii) to identify $K$ and $\KH$ of valuation rings.
\end{proof}

\section{cdh descent for perfect schemes}
In this section we give a new proof that $K[\tfrac1p]$ satisfies cdh descent in characteristic $p$. In the following lemma the notation $A_\sub{perf}:=\indlim_\phi A$ denotes the colimit perfection of any $\bb F_p$-algebra $A$:

\begin{lemma}\label{lemma_radicial_maps}
Let $A$ be an $\bb F_p$-algebra. Then the canonical map $K(A)[\tfrac1p]\to K(A_\sub{perf})[\tfrac1p]$ is an equivalence.
\end{lemma}
\begin{proof}
Since $A_\sub{perf}$ is unchanged if we replace $A$ by $A_\sub{red}$, and $A\to A_\sub{red}$ induces an equivalence on $K[\tfrac1p]$ (since $K[\tfrac1p]$ is invariant under locally nilpotent ideals in characteristic $p$), we may replace $A$ by $A_\sub{red}$ and so assume that the Frobenius $\phi$ is injective on $A$. It is then clearly enough to show that $\phi(A)\to A$ induces an isomorphism on $K[\tfrac1p]$. Changing notation and taking a filtered colimit, this reduces to the following claim:
\begin{quote}
Let $A\subseteq B$ be an extension of rings such that $B=A[b]$ for some $b\in A$ such that $b^p\in A$. Then $K(A)[\tfrac1p]\to K(B)[\tfrac1p]$ is an equivalence.
\end{quote}

Let $A\subseteq B$ be as in the claim and set $a:=b^p\in A$. The map $A[X]/X^p-a\to B$, $X\mapsto b$ is surjective and has nilpotent kernel (any element $f$ in the kernel satisfies $f^p\in A$, but $A\to B$ is injective), hence induces an equivalence on $K[\tfrac1p]$ by nil-invariance of $K[\tfrac1p]$ in characteristic $p$. Thus we have reduced to the special case $B=A[X]/X^p-a$.

To treat the special case we consider the base change square 
\[\xymatrix@=1.5cm{
A[Y]/Y^p-a\ar@{^(->}[r]^{g'}&B[Y]/Y^p-a\\
A\ar@{^(->}[u]^f\ar@{^(->}[r]_g& B\ar@{^(->}[u]_{f'}
}\]
in which all maps are finite flat. The map $f'$ is split by the $B$-algebra homomorphism $B[Y]/Y^p-a\to B$, $Y\mapsto X$, whose kernel is nilpotent (similarly to the previous paragraph); therefore, again using nil-invariance of $K[\tfrac1p]$, we have shown that $f'^*$ is an equivalence on $K[\tfrac1p]$. But $f'_*f'^*=p$, so $f'_*$ is also an equivalence on $K[\tfrac1p]$. The same argument, swapping the roles of $X$ and $Y$, shows that $g'^*$ and $g'_*$ are equivalences on $K[\tfrac1p]$.

Base change for algebraic $K$-theory tells us that $g^*f_*=f'_*g'^*:K(A[Y]/Y^p-a)\to K(B)$, whence we deduce that $g^*f_*$ is an equivalence on $K[\tfrac1p]$. But $f_*$ is surjective on homotopy groups since $f_*f^*=p$, whence it follows that both $g^*$ and $f_*$ are equivalences on $K[\tfrac1p]$.
\end{proof}

\begin{remark}\label{remark_perf}
The higher $K$-groups $K_n$, $n\ge1$, of any perfect $\bb F_p$-algebra are known to be uniquely $p$-divisible \cite[Corol.~5.5]{Kratzer1980}. Therefore the previous lemma implies that $K_n(A)[\tfrac1p]\cong K_n(A_\sub{perf})$ for $n\ge1$.
\end{remark}

The previous lemma may be combined with results of Bhatt--Scholze and Tamme to prove a  cdh descent result for $K[\tfrac1p]$ in characteristic $p$ (which is not really new; see Remark \ref{remark_pro_cdh}), namely Theorem~\ref{theorem_cdhp_intro} from the introduction:

\begin{theorem}\label{theorem_cdh_descent}
Let
\[\xymatrix{
Y'\ar[d]\ar[r]^i & X'\ar[d] \\
Y\ar[r] & X
}\]
be an abstract blow-up square of schemes on which $p$ is nilpotent; assume that $X$ is quasi-compact quasi-separated and that the morphisms $X'\to X$ and $Y\into X$ are of finite presentation. Then the resulting square 
\[\xymatrix{
K(X)[\tfrac1p]\ar[d]\ar[r] & K(X')[\tfrac1p]\ar[d] \\
K(Y)[\tfrac1p]\ar[r] & K(Y')[\tfrac1p]
}\]
is homotopy cartesian.
\end{theorem}
\begin{proof}
Using once again nil-invariance of $K[\tfrac1p]$ we may apply $-\otimes_\bb Z\bb Z/p\bb Z$ to the diagram and so reduce to the case of $\bb F_p$-schemes. By Lemma \ref{lemma_radicial_maps} it is then enough to prove the stronger statement that
\[\xymatrix{
K(X_\sub{perf})\ar[d]\ar[r] & K(X'_\sub{perf})\ar[d] \\
K(Y_\sub{perf})\ar[r] & K(Y'_\sub{perf})
}\]
is homotopy cartesian. To do this we appeal to Tamme's excision condition \cite[Thm.~18]{Tamme2017}, which states that it is enough to check that the square of stable $\infty$-categories
\[\xymatrix{
D_\sub{qc}(X_\sub{perf})\ar[d]\ar[r] & D_\sub{qc}(X'_\sub{perf})\ar[d] \\
D_\sub{qc}(Y_\sub{perf})\ar[r] & D_\sub{qc}(Y'_\sub{perf})
}\]
has the following two properties:
\begin{enumerate}[(a)]
\item it is a pull-back of $\infty$-categories;
\item the right adjoint $Ri_*:D_\sub{qc}(Y'_\sub{perf})\to D_\sub{qc}(X'_\sub{perf})$ is fully faithful.
\end{enumerate}
Both these properties about perfect schemes are due to Bhatt--Scholze. Firstly, the pull-back condition (a) is  \cite[Corol.~5.28]{BhattScholze2017}. Secondly, condition (b) is a consequence of $Li^*Ri_*=\op{id}$, which follows from the fact that if $S\to S/I$ is a surjection of perfect $\bb F_p$-algebras, then $S/I\dotimes_S S/I\quis S/I$ \cite[Lem.~3.16]{BhattScholze2017}.
\end{proof}

\begin{remark}\label{remark_pro_cdh}
We review some previous proofs of cdh descent properties of $K$-theory.
\begin{enumerate}
\item Firstly, if $X$ were assumed to be of finite type over a perfect field of characteristic exponent $p\ge1$ admitting resolution of singularities in a strong sense, then Haesemeyer's argument \cite{Haesemeyer2004} would show that $\KH$ satisfies cdh descent, thereby giving the homotopy cartesian square for $K[\tfrac1p]$ replaced by $\KH$; but then one can invert $p$ and use that $\KH[\tfrac1p]=K[\tfrac1p]$ when $p$ is nilpotent.

Alternatively, under the same assumptions on $X$ and the base field, trace methods were used by the second author \cite{Morrow_pro_cdh_descent} to establish pro cdh descent, namely that the square of pro spectra 
\[\xymatrix{
\{K(Y'_r)\}_r\ar[d]\ar[r] & K(X')\ar[d] \\
\{K(Y_r)\}_r\ar[r] & K(X)
}\]
is homotopy cartesian, where $Y_r$ is the $r^\sub{th}$-infinitesimal thickening of $Y$ in $X$ (and similarly for $Y'$). Since $K[\tfrac1p]$ is invariant under nilpotent ideals when $p$ is nilpotent, inverting $p$ then yields the same homotopy cartesian square of Theorem \ref{theorem_cdhp_intro}.

\item Alternatively, assuming only that $X$ is Noetherian and of finite Krull dimension, Cisinski showed that $\KH$ satisfies cdh descent \cite{Cisinski2013}, again giving the desired homotopy cartesian square for $\KH$ rather than $K[\tfrac1p]$.
\item Kerz--Strunk--Tamme have proved pro cdh descent \cite{KerzStrunkTamme2016} for Noetherian schemes without any hypotheses on resolution of singularities. As in (i), inverting $p$ then gives the homotopy cartesian square of Theorem \ref{theorem_cdhp_intro} assuming that $X$ is Noetherian.

 \item Most recently, Land--Tamme have substantially clarified the theory of cdh descent for arbitrary ``truncating invariants'' \cite{LandTamme2018}.
\end{enumerate}
\end{remark}

\bibliographystyle{acm}
\bibliography{../Bibliography}

\noindent Shane Kelly \hfill Matthew Morrow \\
Department of Mathematics, \hfill CNRS \& IMJ-PRG,\\
Tokyo Institute of Technology, \hfill  SU -- 4 place Jussieu,\\
2-12-1 Ookayama, Meguro-ku, \hfill Case 247,\\
Tokyo 152-8551, Japan \hfill 75252 Paris\\
{\tt shanekelly@math.titech.ac.jp}\hfill {\tt matthew.morrow@imj-prg.fr}

\end{document}